\documentclass[10pt,a4paper]{article} 

\usepackage[latin1]{inputenc}  
\usepackage[babel=true]{csquotes}
\usepackage{amsmath,amscd,amssymb,latexsym,graphicx,color,mathrsfs} 
 
 \usepackage[all]{xy}
\usepackage{hyphenat}
\newtheorem{theorem}{Theorem}[section]

\newtheorem{proposition}[theorem]{Proposition}

\newtheorem{lemma}[theorem]{Lemma}
\newtheorem{definition}[theorem]{Definition}
\newtheorem{remark}[theorem]{Remark}

\newtheorem{nota}[theorem]{Notation}
\newtheorem{example}[theorem]{Example}

\numberwithin{equation}{section} 
 
\newcommand{\cqfd}{\hfill{\small $\Box$}} 
 \newenvironment{proof}[1][]{{\bf Proof #1 : }}{\hfill \cqfd} 
 
\reversemarginpar

\newcommand{\To}{\longrightarrow}

\newcommand{\id}{\mbox{id}}

\newcommand{\gr}{\mathscr{G}}
\newcommand{\go}{\mathscr{G} ^{(0)}}
\newcommand{\hr}{\mathscr{H}}
\newcommand{\ho}{\mathscr{H} ^{(0)}}

\newcommand{\Sp}{\mathscr{S}_{p}}

\newcommand{\F}{\mathscr{F}}

\newcommand{\Qp}{{\mathbb Q}_p}
\newcommand{\Sol}{\mathscr{S}}

\newcommand\tgt[1]{{}^{T}\kern-1pt #1}
\newcommand\adi[1]{{}^{ad}\kern-1pt #1}

\newcommand{\im}{\mathop{\mathrm{im}}\nolimits}

  \def\RR{{\mathrm{R}}}

 \def\NN{{\mathbb{N}}} 
 \def\QQ{{\mathbb{Q}}} \def\RR{{\mathbb{R}}}

 \def\ZZ{{\mathbb{Z}}}

\def\cA{{\mathcal{A}}}

  \def\cU{{\mathcal{U}}}
\def\cV{{\mathcal{V}}}  
 
  \title{Groupoids, equivalence bibundles and bimodules for noncommutative solenoids}

  \author{Paulo Carrillo Rouse and Laurent Guillaume}
  
\begin{document}

\maketitle

\bigskip
\everymath={\displaystyle}

\begin{abstract}
\noindent 
Let $p$ be a prime number and $\Sp$ the $p$-solenoid. For $\alpha\in \mathbb{R}\times \mathbb{Q}_p$ we consider in this paper a naturally associated action groupoid $S_\alpha:=\ZZ [1/p]\ltimes_\alpha \Sp \rightrightarrows \Sp$ whose $C^*-$algebra is a model for the noncommutative solenoid $\cA_\alpha^\mathscr{S}$ studied by Latremoli\`ere and Packer. Following the geometric ideas of Connes and Rieffel to describe the Morita equivalences of noncommutative torus using the Kronecker foliation on the torus, we give an explicit description of the geometric/topologic equivalence bibundle for groupoids $S_\alpha$ and $S_\beta$ whenever $\alpha,\beta\in \RR\times \QQ_p$ are in the same orbit of the $GL_2(\ZZ[1/p])$ action by linear fractional transformations. As a corollary, for $\alpha,\beta\in \RR\times \QQ_p$ as above we get an explicit description of the imprimitivity bimodules for the associated noncommutative solenoids.
\end{abstract}

\tableofcontents

\section*{Introduction}

One of the first and most studied examples of noncommutative spaces in noncommuative geometry is the so called noncommutative torus $A_\theta$ for non rational $\theta\in \mathbb{R}$ (the rational case can be as well defined but it is less interesting as a "noncommutative" space). The $C^*$-algebra $A_\theta$ admits many different descriptions (up to isomorphism and up to Morita as we will recall in the text). For example it can be realized as the $C^*-$algebra of the transformation groupoid associated to the action of $\ZZ$ on $S^1$ by rotations by $\theta$, we denote by $\mathbb{T}_\theta:=\ZZ\ltimes_\theta S^1 \rightrightarrows S^1$ this Lie groupoid.

One of main notions of equivalence for $C^*$-algebras in noncommutative geometry is the one introduced by Rieffel and called strong Morita equivalence. In \cite{Rieffel}, Rieffel showed that $A_\theta$ and $A_{\theta'}$ are strongly Morita equivalent if and only if $\theta$ and $\theta'$ are in the same orbit of the $GL_2(\ZZ)$ action on irrational numbers by linear fractional transformations. Later on, with Connes (see Connes book \cite{Concg} section $II.8.\beta$), they gave an even more geometric description of these equivalences for noncommutative torus by exploring the restriction to complete transversals on the holonomy groupoid associated to the so called Kronecker foliation of the torus by irrational slopes $\theta$. This geometric approach allows them in particular to give a precise description of the Morita equivalence bimodules.

For groupoids (Lie or locally compact) there is a topological/geometrical notion of Morita equivalence given by the so called equivalence bibundles. This very studied equivalence in geometry/topology allows to properly formalize the study of geometrical/topological stacks. Now, given two Morita equivalent groupoids it is known that the corresponding $C^*$-algebras are strongly Morita equivalent (Renault's equivalence theorem, see for instance 2.70 in \cite{Williams}). In fact, whenever an appropriate/explicit description of an equivalence bibundle is given there is a very natural construction of the associated Morita equivalence bimodule. 

For $\theta,\theta'$ irrationals in the same $GL_2(\ZZ)$-orbit, the Connes-Rieffel approach gives a good glimpse of the explicit description of an equivalence bibundle. By explicit we mean a precise bibundle space, explicit left and right moment maps and explicit left and right groupoid actions. The motivation for the content of the present article was to give an explicit description for  the equivalence bibundles for $\mathbb{T}_\theta$ and $\mathbb{T}_{\theta'}$ for $\theta,\theta'$ irrationals in the same $GL_2(\ZZ)$-orbit. For noncommutative torus, this expected folklore result was for us the starting point to go beyond on the exploration for the case of noncommutative solenoids. Let us explain with more detail the content of this work.

Let $p$ be a prime number and $\Sp$ the associated solenoid. Given $\alpha\in \RR\times \QQ_p $ there is a transformation groupoid 
\begin{equation}
S_\alpha:=\ZZ[1/p]\ltimes_\alpha \Sp \rightrightarrows \Sp
\end{equation}
whose $C^*$-algebra $C^*(S_\alpha)$ is a model (as justified in theorem \ref{thmncsolenoidgrpd}) for the so called "Noncommutative solenoid associated to $\alpha$". These groupoids $S_\alpha$ are called here solenoidal groupoids. Now, noncommutative solenoids have been extensively studied by Latremoli\`ere and Packer \cite{LP1,LP2,LP3}, in their works they have accomplished several interesting questions, for example the classifications of twistings, the computation of their $K$-theory groups and the construction of explicit equivalence bimodules in some cases, just for mention some of them (see also Lu's work \cite{Lu}). Also, very promising applications to the so called Gromov-Hausdorff and spectral propinquity have been explored, see for instance \cite{LP4,FLP}.

Our main theorem (statement \ref{maintheorem} below) gives a very precise and explicit description of the equivalence bibundle between two solenoidal groupoids $S_\alpha$ and $S_\beta$ for $\alpha,\beta\in \RR\times \QQ_p$ in the same orbit of the $GL_2(\ZZ[1/p])$ action by linear fractional transformations.

As a corollary we obtain explicit equivalence bimodules between the respective noncommutative solenoids $C^*(S_\alpha)$ and $C^*(S_\beta)$. These equivalence bimodules recover, unify and generalize, with a geometric groupoid approach, all previously known cases in \cite{LP2,LP3,Lu}.

\section{Locally compact groupoids and bibundle equivalences}

In this section we put some basic preliminaries on groupoids and their equivalences. For more details see the recent book of Williams \cite{Williams}.

Let us recall what a groupoid is:

\begin{definition}
A $\it{groupoid}$ consists of the following data:
two sets $\gr$ and $\go$, and maps
\begin{itemize}
\item[(1)]  $s,r:\gr \rightarrow \go$ 
called the source map and target map respectively,
\item[(2)]  $m:\gr^{(2)}\rightarrow \gr$ called the product map 
(where $\gr^{(2)}=\{ (\gamma,\eta)\in \gr \times \gr : s(\gamma)=r(\eta)\}$),
\end{itemize}
together with  two additional  maps, $u:\go \rightarrow \gr$ (the unit map) and 
$i:\gr \rightarrow \gr$ (the inverse map),
such that, if we denote $m(\gamma,\eta)=\gamma \cdot \eta$, $u(x)=x$ and 
$i(\gamma)=\gamma^{-1}$, we have 
\begin{itemize}
\item[(i).]$r(\gamma \cdot \eta) =r(\gamma)$ and $s(\gamma \cdot \eta) =s(\eta)$.
\item[(ii).]$\gamma \cdot (\eta \cdot \delta)=(\gamma \cdot \eta )\cdot \delta$, 
$\forall \gamma,\eta,\delta \in \gr$ whenever this makes sense.
\item[(iii).]$\gamma \cdot x = \gamma$ and $x\cdot \eta =\eta$, $\forall
  \gamma,\eta \in \gr$ with $s(\gamma)=x$ and $r(\eta)=x$.
\item[(iv).]$\gamma \cdot \gamma^{-1} =u(r(\gamma))$ and 
$\gamma^{-1} \cdot \gamma =u(s(\gamma))$, $\forall \gamma \in \gr$.
\end{itemize}
For simplicity, we denote a groupoid by $\gr \rightrightarrows \go $. A strict morphism $f$ from
a  groupoid   $\hr \rightrightarrows \ho $  to a groupoid   $\gr \rightrightarrows \go $ is  given
by  maps in 
\[
\xymatrix{
\hr \ar@<.5ex>[d]\ar@<-.5ex>[d] \ar[r]^f& \gr \ar@<.5ex>[d]\ar@<-.5ex>[d]\\
\ho\ar[r]^f&\go
}
\]
which preserve the groupoid structure, i.e.,  $f$ commutes with the source, target, unit, inverse  maps, and respects the groupoid product  in the sense that $f(h_1\cdot h_2) = f (h_1) \cdot f(h_2)$ for any $(h_1, h_2) \in \hr^{(2)}$.

\end{definition}

In this paper we will only deal with locally compact groupoids, that is, 
a groupoid in which $\gr$ and $\go$ are locally compact Hausdorff spaces, and $s,r,m,u$ are continuous maps. For two subsets $U$ and $V$ of $\go$,
 we use the notation
$\gr_{U}^{V}$ for the subset 
\[
\{ \gamma \in \gr : s(\gamma) \in U,\, 
r(\gamma)\in V\} .
\]

\subsection*{ Groupoid equivalence bibundles:}

\begin{definition}
Let $\gr \rightrightarrows \go$ and  
$\hr \rightrightarrows \ho$ be two locally compact groupoids with open range maps.

  Let $\gr \rightrightarrows \go$ and  
$\hr \rightrightarrows \ho$ be two locally compact groupoids.  A $\gr$-$\hr$ bibundle equivalence
is a left $\gr$-bundle over $\ho$
which is also a right $\hr$-bundle over $\go$,  formally denoted by
\[
\xymatrix{
\gr \ar@<.5ex>[d]\ar@<-.5ex>[d]&P \ar@{->>}[ld]_-{l} \ar[rd]^-{r}&\hr \ar@<.5ex>[d]\ar@<-.5ex>[d]\\
\go&&\ho.
}
\]
and such that 
\begin{enumerate}
\item $P$ is a free and proper left $\gr$-space.
\item $P$ is a free and proper right $\hr$-space.
\item The $\gr$- and $\hr-$actions commute.
\item The moment maps $l$ and $r$ are open and induce homeomorphisms
\begin{equation}
P/\ho \to \go
\end{equation}
and 
\begin{equation}
\go\backslash P \to \ho.
\end{equation}
\end{enumerate}

\end{definition}

As mentioned in the introduction, two equivalent groupoids in the sense above define Morita equivalent $C^*$-algebras (reduced or maximal). This is the content of Renault's equivalence theorem (theorem 2.70 in \cite{Williams} or Renault's original source \cite{Ren}). But even more, the precise description of an equivalence bibundle gives a precise description of the associated Morita bimodule (ref.cit. or \ref{Rentheorem} below.)

Equivalence bibundles are a particular case of the notion of generalized morphisms. In the context of foliations and Lie groupoids this last notion seemed to appear first in the work of Hilsum and Skandalis, \cite{HS}. 

\subsection*{Examples of equivalence bibundles}

\begin{enumerate}

  \item (Strict morphism of groupoids)
    Let $f : \hr \to \gr$ be a strict morphism of groupoids and
    $$P_f = \ho \times_{\go} \gr = \{ (h, \gamma) | f(h) = s(\gamma) \}$$
    $P_f$ is a principal $\gr$-$\hr$-bibundle for the right $\gr$-action given by composition $(h, \gamma)\cdot \gamma' = (h, \gamma\gamma')$ and left $\hr$-action given by
    $$\eta \cdot (s(\eta),\gamma) = (r(\eta), f(\eta) \cdot \gamma )$$
    A principal $\gr$-$\hr$-bibundle $P_f$ comes from a strict morphism $f$ if and only if the map $\pi$ admits a continuous section.

\item(Transversals of foliated manifolds)\label{transversals}
      A \textsl{complete transversal} of a foliated manifold $(M, \F)$ is a $C^1$-immersion $T:N\rightarrow M$ of a manifold N of dimension $q=codim(\F)$ which is transversal to the leaves of $\F$ and such that $T(N)$ meets any leaf of $\F$ in at least one point. 
      The holonomy groupoid $Hol_T(M,\F)$ \textsl{reduced to $T$} is then the pull-back groupoid 
\[
\xymatrix{
  Hol_T(M,\F) \ar@<.5ex>[d]\ar@<-.5ex>[d] \ar[r]^{^*T^*}& Hol(M,\F) \ar@<.5ex>[d]\ar@<-.5ex>[d]\\
N\ar[r]^T& M
}
\]   
which arrows are holonomy classes $[(x,\gamma, y)]$ for $x,y\in N$ and $\gamma$ a path from $T(x)$ to $T(y)$ inside the leaves of $\F$. 

The bibundle $P_{^*T^*}^{-1} = \{ (y,\gamma) \in N\times_r Hol(M), T(y) = r(\gamma)\} \simeq Hol(M,\F)^N$ is a $Hol_T(M,\F)$-$Hol(M,\F)$ equivalence bibundle with left action given by $[(x, \gamma, y)]\cdot (y,\gamma') = (x,\gamma\cdot\gamma')$
and right action $(y,\gamma)\cdot\gamma' = (y, \gamma\cdot\gamma')$.

Composition of bibundles $P_{^*T^*}^{-1} \times_{Hol(M,\F)} P_{^*T'^*} \simeq Hol(M,\F)_{N'}^N$ is a Morita equivalence between $Hol_T(M,\F)$ and $Hol_{T'}(M,\F)$. The equivalence bibundle is the quotient of pairs of composable arrows $(\gamma, \gamma') \in Hol(M,\F)^N \times Hol(M,\F)_{N'}$ by forgetting midterm $s(\gamma) = r(\gamma') \in M$.

If $T,T'$ are complete transversals, $T'' = T\sqcup T' : N \sqcup N' \rightarrow M$ is also a complete transversal and the inclusions $Hol_T(M,\F)\hookrightarrow Hol_{T''}(M,\F)$ and $Hol_{T'}(M,\F)\hookrightarrow Hol_{T''}(M,\F)$ induce another Morita equivalence between $Hol_T(M,\F)$ and $Hol_{T'}(M,\F)$. 

\item(Reduced groupoids with open maps)\label{reduced}

More generally, suppose that $G$ is a locally compact Hausdorff groupoid with open range and source maps and let $A$ and $B$ be \textsl{open} subsets of $G^{(0)}$ meeting every $G$-orbit in $G^{(0)}$. Then $P = G^A_B$ is an equivalence between the reduced groupoids $G_A^A$ and $G_B^B$. The result still holds if $A$ and $B$ are not open but the restrictions of $r$ and $s$ to $G_A^B$ remain open. 

\[
\xymatrix{
G_A^A \ar@<.5ex>[d]\ar@<-.5ex>[d]&G_B^A \ar@{->>}[ld]_r \ar[rd]^s &G_B^B \ar@<.5ex>[d]\ar@<-.5ex>[d]\\
A&&B.
}
\]

If $A$ and $B$ are open restrictions of $s$ and $r$ to $G_B^A$ are open maps since $G_A^B$ is open in $G$.  Associativity of composition in $G$ implies commutative actions $\gamma_1(\gamma_2\gamma_3) = (\gamma_1\gamma_2)\gamma_3$. To show that $r : P \rightarrow A$ induces a homeomorphism $P/G_B^B \simeq A$, first notice that the relation $G_A^B = AGB$ implies that $r$ is onto. Then let $\gamma,\gamma'\in P=G_A^B$ such that $r(\gamma) = r(\gamma')$. Write $\gamma' = \gamma(\gamma^{-1}\gamma')$ and note that $\gamma^{-1}\gamma' \in G_B^B$ so that $r$ is injective on $P/G_B^B \simeq A$. For more details see example 2.39 in \cite{Williams}, page 46.

\end{enumerate}

\section{Solenoids as homogeneous spaces}

The Pontryagin dual of the group $\ZZ[1/p]$ of $p$-adic rationals is the $p$-solenoid given by:
$$\mathscr{S}_p = \left \{ (z_n)_{n\in\NN} \in \mathbb{T}^\NN \,,\, \forall n\in\NN \,: z_{n+1}^p = z_n \right \}$$
with the induced topology from the injection $\mathscr{S}_p \hookrightarrow \mathbb{T}^\NN$. 

\subsection{Exact sequences for solenoids}

Let $\omega : \RR \rightarrow \Sp$ be the standard "winding line" defined for any $t\in\RR$ by $ \omega(t) = ( e^{2i\pi t / p^n} )_{n\in\NN}$. For  $r=\sum_{j=k}^\infty a_jp^j \in\Qp$ denote $\{r\} =  \sum_{j=k}^{-1} a_jp^j$ the fractional part of $r$ and let $\zeta : \Qp \rightarrow \Sp$  be defined by the sequence $ \zeta(r) = (e^{2i\pi \{r / p^n\}})_{n\in\NN}$. We thus may define $\pi : \RR\times\Qp \rightarrow \Sp$ by $\pi(t,r) = w(t)\cdot \zeta(-r)$ and $\delta : \ZZ[1/p] \rightarrow \RR\times\Qp$ by $\delta(n) = (n, n)$. 

\begin{proposition}
The maps $\pi$ and $\Delta$ induce the exact sequences of topological groups 
\[
  \xymatrix{
    1 \ar[r] & \ZZ \ar[r]^\delta &  \RR \times \ZZ_p \ar[r]^{\pi} &  \Sp \ar[r] & 1\\
  }
\]
\[
  \xymatrix{
     1 \ar[r] & \ZZ[1/p] \ar[r]^\delta & \RR \times \Qp \ar[r]^{\pi} &  \Sp \ar[r] & 1
  }
\]
so that $\Sp$ is a homogeneous space for the action $\rho : \RR \times \Qp \xrightarrow{\pi} \text{Homeo}(\Sp)$ by translation.
\end{proposition}

\begin{proof}
  The result follows from the topological properties of $p$-adic integers $\ZZ_p$ and rationals $\Qp$ detailed in the next three lemmas.
\end{proof}

\subsection{$\Sp$ is an homogeneous space for the actions of $\RR \times \ZZ_p$ and $\RR \times \Qp$}

\begin{lemma}
The morphisms $\varphi_{n} : p^{-n}\ZZ_p \rightarrow p^{-(n+1)}\ZZ_p$ of the inductive system $\Qp = \varinjlim p^{-n}\ZZ_p$  are open homeomorphisms to their image and each $\RR \times p^{-n}\ZZ_p$ is acting continuously on $\Sp$.
\end{lemma}

\begin{proof}
$\varphi_n$ as a group morphism between projective limits is defined by the compatible family of inclusions
\[
  \label{projective_diagram}
  \xymatrix{
    p^{-n}\ZZ_p \ar[d]^{\varphi_n} & \dots \ar[r]^{\phi_{n,1}} &  p^{-n}\ZZ/p^{-n+1} \ar[d]^{i_1} \ar[r]^{\phi_{n,0}} &  p^{-n}\ZZ/p^{-n} \ar[d]^{i_0} &  \\
    p^{-(n+1)}\ZZ_p  & \dots \ar[r]^{\phi_{n+1,2}} &  p^{-(n+1)}\ZZ/p^{-n+1} \ar[r]^{\phi_{n+1,1}} &  p^{-(n+1)}\ZZ/p^{-n} \ar[r]^{\phi_{n+1,0}} &  p^{-(n+1)}\ZZ/p^{-(n+1)} 
  }
\]

with first projection map $p^{-n}\ZZ_p \rightarrow  p^{-(n+1)}\ZZ/p^{-(n+1)}$ obtained by composition with any inclusion map $i_k$ from first row and morphisms $\phi_{n+1, k} = \mod p^k$ to the bottom. 

To show that $\varphi$ is injective, observe that an element $g$ in a projective limit $G = \varprojlim G_k$ is 0 if and only if all its projections $g_k = p_k(g)$ are 0.
Indeed $u_k:\{pt\}\rightarrow 0_k \in G_k$ are compatible maps in the set category and define a \textsl{unique} (set) morphism $u:\{pt\}\rightarrow 0 \in G$ so that $g_k = 0$ is the only image possible for $g=0$.
So for $g\in p^{-n}\ZZ_p$ non-zero there exists a non-zero $g_k \in p^{-k}\ZZ/p^{-(k+1)}$ and $i_k(g_k)$ is non-zero so that $\varphi_n(g)$ cannot be zero. The homomorphism $\varphi_n$ thus has a trivial kernel and is injective. 
Being a bijective continuous map from a compact space to a Hausdorff space,  $\varphi_n$ is a homeomorphism on its image.  

To show that $\varphi_n$ has open range, we verify that $\varphi_n(p^{-n}\ZZ_p) = \ker q_1 = q_1^{-1}(\{0\})$ where $(q_k)$ denote the projections from $p^{-(n+1)}\ZZ_p$ to $p^{-(n+1)}\ZZ/p^{-(n+1)+k}$. The result will follow as $q_1$ is continuous and $\{0\}$ is open in the discrete topology of $\ZZ/p^n\ZZ$.

First commutativity of the diagram \ref{projective_diagram} implies that $\varphi_n(p^{-n}\ZZ_p) \subset \ker q_1$.
Now $f_k = i_k^{-1} \circ q_k : \ker q_1 \rightarrow p_k( p^{-n}\ZZ_p)$ is a compatible family of maps that induce a map $f:\ker q_1 \rightarrow p^{-n}\ZZ_p$ which is an inverse for $\varphi_n$. Indeed $q_k\circ\varphi_n\circ f = i_k \circ p_k \circ f = i_k \circ i_k^{-1} \circ q_k = q_k$ so that $\varphi_n \circ f = \id_{\ker q_1}$ and $p_k\circ f \circ \varphi_n = i_k^{-1} \circ q_k \circ \varphi_n = i_k^{-1} \circ i_k \circ p_k = p_k$ so that $f \circ \varphi_n = \id$.

Each $\RR \times p^{-n}\ZZ_p$ is acting continuously on $X_n = \left \{ (z_{k-n})_{k\in\NN} \in \mathbb{T}^{\NN-n} \,,\, \forall k\in\NN \,: z_{k+1-n}^p = z_{k-n} \right \}$ by $\rho_n(t, g, z) = \omega(t) \cdot \zeta(-g) \cdot z\in \mathbb{T}^{\NN-n}$ with the induced topology from $\mathbb{T}^{\NN-n}$ and there is a natural (unique) isomorphism $X_n \simeq \Sp$ as cofinal limits induced from the inclusion $\Sp = X_0 \subset X_n$.

\end{proof}

\begin{lemma}
$\Qp$ is a locally compact Hausdorff topological group, with open inclusion maps $j_n : p^{-n}\ZZ_p \rightarrow \Qp$.
\end{lemma}

\begin{proof}
For any $n\in \NN$,  $p^{-n}\ZZ_p$ is open in $p^{-(n-1)}\ZZ_p$, $p^{-(n-2)}\ZZ_p,\dots$ by continuity of $\varphi_{n-1}, \varphi_{n-2}$ and is open in $p^{-(n+1)}\ZZ_p, p^{-(n+2)}\ZZ_p, \dots$ since $\varphi_{n}$, $\varphi_{n+1}$ are open. Thus $p^{-n}\ZZ_p$ is open for the inductive topology in $\Qp$.
The same argument is true for any open subset of $p^{-n}\ZZ_p$ and $j_n : p^{-n}\ZZ_p \rightarrow \Qp$ are thus open.

$\Qp = \varinjlim_{\varphi_n} p^{-n}\ZZ_p$ where each $\varphi_{n}$ has open range and each $p^{-n}\ZZ_p$ is a locally compact Hausdorff space. Thus $\Qp$ is a locally compact Hausdorff space. For $x,y\in\Qp$ choose $n$ large so that $x,y\in p^{-n}\ZZ_p$. Group operation is continuous on $p^{-n}\ZZ_p$ which is open in $\Qp$ so that composition $\Qp \times \Qp \rightarrow \Qp$ is continuous and $\Qp$ is a topological group.  

\end{proof}

\begin{lemma}
$\Sp$ is a homogeneous space for the action $\rho : \RR \times \Qp \xrightarrow{\pi} \text{Homeo}(\Sp)$ by translation
\end{lemma}

\begin{proof}
For $r\in\Qp$ take $n$ large enough so that $r\in p^{-n}\ZZ_p$ and define $\rho(t, r, z) = \rho_n(t, r, z)$ from the action $\rho_n : \RR \times p^{-n}\ZZ_p \times \Sp \rightarrow \Sp$ of $\RR \times p^{-n}\ZZ_p$ on $\Sp$.
From the compatibility of the actions $\rho_n$ with the inclusions $j_n$ the action $\rho : \RR\times \Qp \times \Sp \rightarrow \Sp$ is well defined. Associativity follows in the same way.

The action $\rho$ is continuous as each $p^{-n}\ZZ_p$ is open in $\Qp$ and the actions $\rho_n$ are continuous and compatibles, so that for any open set $\Omega \in \Sp$, $\rho^{-1}(\Omega) = \bigcup \rho_n^{-1}(\Omega)$ is a union of open sets hence open.

The action $\rho$ is also transitive as it is already transitive on $\RR \times \ZZ_p$ : for $z = (z_k)_{k\in\NN}\in \Sp$ choose $t\in\RR$ such that $e^{2i\pi t} = z_0$. Then $y = \omega(-t)\cdot z$ satisfies $(y_k^{p^k})_{k\in\NN} = 1$ so that $y=\zeta(-r)$ with $r\in\ZZ_p$ and $z = \omega(t)\zeta(-r)$. Similar computation shows that $\delta(\ZZ) = \{(k, k) \in \ZZ_p\times\RR \}$ is acting trivialy for the $\RR \times \ZZ_p$ action and that $\delta(\ZZ[1/p]) = \bigcup \delta(p^{-n}\ZZ)$ is the isotropy group of the $\RR \times \Qp$ action on $\Sp$.
 
\end{proof}

\section{Equivalence bibundles of noncommutative solenoids}

\subsection{Solenoidal groupoids $S_\alpha$ and $\mathscr{S}_\alpha$ }


Let $p$ be a prime number, $\Sp$ the associated solenoid and $\alpha \in\RR\times\Qp$.\\ 
We introduce the full solenoidal groupoid  
\begin{equation}
\Sol_\alpha = (\RR\times\Qp)\ltimes_{\alpha}\Sp^2 \rightrightarrows \Sp^2
\end{equation}
as the action groupoid defined by the $\RR\times\Qp$ action $\pi$
$$q\cdot (x,y) = (\pi(q)x, \pi(q\alpha)y)$$
with $\pi$ acting by translation on $\Sp$ as a $\RR \times \Qp$ homogenous space.

\bigskip

We also introduce the reduced solenoidal groupoid  
\begin{equation}
S_\alpha:=\ZZ[1/p]\ltimes_\alpha \Sp \rightrightarrows \Sp
\end{equation}
defined by the $\ZZ[1/p]$ action on $\Sp$ : $q\cdot z =  \pi(q\alpha)z$

The groupoid structural maps are given as follows:

\begin{itemize}
\item The source and range maps are $s(q,z) = z$, $r(q,z) = q\cdot z$.\\
\item The composition is $(q', z') \cdot (q, z) = ( q' + q , z)$ whenever $q\cdot z = z'$.\\ 
\item The inverse is $(q, z)^{-1} = (-q, q\cdot z)$. 
\end{itemize}


Now, $S_\alpha$ and $\Sol_\alpha$ are equivalent groupoids. Indeed:

  \begin{itemize}
    \item $S_\alpha$ is a subgroupoid of $\Sol_\alpha$ through the groupoid immersion $i: S_\alpha \To \Sol_\alpha$ defined by $i(n,z) = (\delta(n), (1,z))$.


    \item From example 1.3, $\Sol_\alpha^V$, with $V=1\times\Sp \subset \Sp^2$, is an $S_\alpha$-$\Sol_\alpha$ equivalence bibundle.
  \end{itemize}

\subsection{$S_\alpha$-$S_{\alpha^{-1}}$ equivalence bibundle }

Let $\alpha \in\RR\times\Qp$ such that $\pi(\alpha) \in \Sp\setminus \{1\}$. In this section we will give the explicit description of the equivalence bibundle corresponding to the matrix $M = 
  \begin{pmatrix}
    0 & 1  \\
    1 & 0 
  \end{pmatrix}$, that gives $\alpha^{-1}=M\cdot \alpha$.

Denote $P_{\alpha,\alpha^{-1}}$ the bibundle 

\begin{equation}
\xymatrix{
  *+[l]{\ZZ[1/p]\ltimes_\alpha \Sp } \ar@<.5ex>[d]\ar@<-.5ex>[d] & \RR \times \Qp \ar@{->}[ld]_\mu \ar[rd]^\epsilon & *+[r]{\ZZ[1/p]\ltimes_{\alpha^{-1}} \Sp}  \ar@<.5ex>[d]\ar@<-.5ex>[d] \\
  \Sp & & \Sp
}
\end{equation}
where 
\begin{itemize}
\item the left and right moment maps are $\mu(q) = \pi(q\alpha)^{-1}$ and $\epsilon(q) = \pi(q)$,

\item the left action is $(n,z)\cdot q = n+q$, given $z = \pi((q+n)\alpha)^{-1}$, and 

\item the right action is $q \cdot (n,z) = q + n\alpha^{-1}$, given $z = \pi(q)$.
\end{itemize}

\begin{proposition}

  $P_{\alpha,\alpha^{-1}}$ is an equivalence bibundle between $S_\alpha$ and $S_{\alpha^{-1}}$ 

\end{proposition}

\begin{proof}
Let $\Sol_\alpha = (\RR\times\Qp)\ltimes_{\alpha}\Sp^2$ be the full solenoidal groupoid defined by the $\RR\times\Qp$-action 
$$q\cdot (x,y) = (\pi(q)x, \pi(q\alpha)y)$$

The horizontal and vertical subspaces $H=\{(x,1), x\in \Sp\}$, $V=\{(1,y), y\in \Sp\}$ are closed transversals of $\Sp^2$ and $(\Sol_\alpha)_V^H = \{(q,1,y), q\in \RR\times\Qp, y = \pi(q\alpha)^{-1}\}$ is an equivalence bibundle between the reduced groupoids $(\Sol_\alpha)_V^V$ and $(\Sol_\alpha)_H^H$ as in \ref{transversals}. 

Moreover, there are groupoid isomorphims (straightforward verification)
$$(\Sol_\alpha)_V^V \simeq_v S(\alpha)$$
given by $v(q,1,y) = (q,y)$,
$$(\Sol_\alpha)_H^H \simeq_h S(\alpha^{-1})$$
given by 
$h(q,x,1) = (q\alpha, x)$, and 
$$(\Sol_\alpha)_V^H \simeq_{p_0} \RR\times\Qp $$
given by 
$p_0(q,1,y=\pi(q\alpha)^{-1}) = q$.


 Thus $(\Sol_\alpha)_V^H$ is a $S_\alpha$-$S_{\alpha^{-1}}$ equivalence bibundle :

\begin{equation}
\xymatrix{
  *+[l]{(\Sol_\alpha)_V^V \simeq_v S_\alpha } \ar@<.5ex>[d]\ar@<-.5ex>[d] & (\Sol_\alpha)_V^H \ar@{->}[ld]_{\tilde{\mu}} \ar[rd]^{\tilde{\epsilon}} & *+[r]{S_{\alpha^{-1}} \simeq_{h^{-1}} (\Sol_\alpha)_H^H }  \ar@<.5ex>[d]\ar@<-.5ex>[d] \\
  \Sp & & \Sp
}
\end{equation}
with structural maps coming from the groupoid structure of $\Sol_\alpha$ and thus explicitly given by:
\begin{itemize}
\item The left and right moment maps are $\tilde{\mu}(q,1,y) = y = \pi(q\alpha)^{-1}$ and $\tilde{\epsilon}(q,1,y) = \pi(q)$.

\item The left action is $(n,z)\cdot (q,1,y) = v(n,1,z)\cdot (q,1,y) = (n+q, 1, z)$, given $z = \pi((q+n)\alpha)^{-1}$

\item The right action is $(q,1,y) \cdot (n,z) = (q,1,y) \cdot h(n\alpha^{-1},z,1) = (q + n\alpha^{-1}, 1, y)$.
\end{itemize}

Now, by composition with the isomorphisms $v,h, p_0$ above we get the desired $S_\alpha - S_{\alpha^{-1}}$ bibundle structure on $P_{\alpha,\alpha^{-1}}$, indeed:
\begin{itemize}
\item $\mu(q) = \tilde{\mu}\circ p_0^{-1}(q) = \tilde{\mu}(q,1, \pi(q\alpha)^{-1}) = \pi(q\alpha)^{-1}$,

\item $\epsilon(q) = \tilde{\epsilon}\circ p_0^{-1}(q) = \tilde{\epsilon}(q,1,\pi(q\alpha)^{-1}) = \pi(q)$, 

\item $(n,z)\cdot q = v(n,1,z) \cdot p_0^{-1}(q) = (n+q, 1, z)$, given $z = \pi((q+n)\alpha)^{-1}$, and

\item $q\cdot (n,z) = p_0^{-1}(q) \cdot h(n\alpha^{-1},z,1) = (q + n\alpha^{-1}, 1, y)$, given $z = \pi(q)$.
\end{itemize}

\end{proof}

\subsection{$S_\alpha$-$S_\beta$ equivalence bibundles $P_M$ for $M\in SL_2(\ZZ[1/p])$ }

Let $M = 
  \begin{pmatrix}
    a & b  \\
    c & d 
  \end{pmatrix}$
  in $SL_2(\ZZ[1/p])$ where we assume $c\neq 0$ (general case described below). Let  
\begin{equation}
\beta = M^{-1}\cdot\alpha = -\frac{b-d\alpha}{a-c\alpha}.
\end{equation}

Denote $P_M$ the bibundle 

\begin{equation}
\xymatrix{
  *+[l]{\ZZ[1/p]\ltimes_\alpha \Sp } \ar@<.5ex>[d]\ar@<-.5ex>[d] & \RR \times \Qp \ar@{->}[ld]_\mu \ar[rd]^\epsilon & *+[r]{\ZZ[1/p]\ltimes_\beta \Sp}  \ar@<.5ex>[d]\ar@<-.5ex>[d] \\
  \Sp & & \Sp
}
\end{equation}

with structural maps given by:
\begin{itemize}
\item Left and right moment maps are $\mu(q) = \pi(q(a-c\alpha)/c)$ and $\epsilon(q) = \pi(q/c)$,

\item Left action is $(n,z)\cdot q = n+q$, given $z = \pi((q+n)(a-c\alpha)/c)$

\item Right action is $q \cdot (n,z) = q + n(a-c\alpha)^{-1}$, given $z=\pi(q/c)$.
\end{itemize}

We have the following theorem.

\begin{theorem}\label{maintheorem}

$P_M$ is an equivalence bibundle between $S_\alpha$ and $S_\beta$ 

\end{theorem}

\begin{proof}
Let $\Sol_\alpha = (\RR\times\Qp)\ltimes_{\alpha}\Sp^2$ be the full solenoidal groupoid.

The vertical and horizontal transversals are now $V=\{(1,y), y\in \Sp\}$ and $H_M \subset \Sp^2$ :
$$H_M = \left \{ (\pi(cq), \pi(aq)) \,,\, q \in \Qp\times\RR \right \}$$

$(\Sol_\alpha)_V^{H_M} = \{(q,1,y), q\in \RR\times\Qp, y = \pi((a-c\alpha)q/c)\}$ is an equivalence bibundle between the reduced groupoids $(\Sol_\alpha)_V^V$ and $(\Sol_\alpha)_{H_M}^{H_M}$ as in \ref{transversals}.

As for the case in the previous section, we have explicit groupoid isomorphims 
$$(\Sol_\alpha)_V^V \simeq_v S_\alpha,$$
$$(\Sol_\alpha)_{H_M}^{H_M} \simeq_h S_\beta$$ 
and 
$$(\Sol_\alpha)_V^{H_M} \simeq_{p_0} \RR\times\Qp $$ 
given respectively by 
$$v(q,1,y) = (q,y),$$ 
$$h(q,cz,az) = (q(a-c\alpha), z)$$ and 
$$p_0(q,1,y=\pi(q(a-c\alpha)/c)) = q.$$ 

Thus $(\Sol_\alpha)_V^{H_M}$ is a $S_\alpha$-$S_\beta$ equivalence bibundle :

\begin{equation}
\xymatrix{
  *+[l]{(\Sol_\alpha)_V^V \simeq_v S_\alpha } \ar@<.5ex>[d]\ar@<-.5ex>[d] & (\Sol_\alpha)_V^{H_M} \ar@{->}[ld]_{\tilde{\mu}} \ar[rd]^{\tilde{\epsilon}} & *+[r]{S_\beta \simeq_{h^{-1}} (\Sol_\alpha)_{H_M}^{H_M} }  \ar@<.5ex>[d]\ar@<-.5ex>[d] \\
  \Sp & & \Sp
}
\end{equation}

with structural maps as follows:
\begin{itemize}
\item Left and right moment maps are $\tilde{\mu}(q,1,y) = y = \pi(q(a-c\alpha)/c)$ and $\tilde{\epsilon}(q,1,y) = \pi(q/c)$

\item Left action is $(n,z)\cdot (q,1,y) = v(n,1,z)\cdot (q,1,y) = (n+q, 1, z)$, given $z = \pi((q+n)(a-c\alpha)/c)$

\item Right action is $(q,1,y) \cdot (n,z) = (q,1,y) \cdot h(n(a-c\alpha)^{-1},z,1) = (q + n(a-c\alpha)^{-1}, 1, y)$, given $z=\pi(q/c)$.
\end{itemize}


Composition with the isomorphisms $v,h, p_0$ gives the desired $S_\alpha - S_\beta$ bibundle structure on $P_M$, indeed:

\begin{itemize}
\item $\mu(q) = \tilde{\mu}\circ p_0^{-1}(q) = \tilde{\mu}(q,1, \pi(q(a-c\alpha)/c)) = \pi(q(a-c\alpha)/c)$,

\item $\epsilon(q) = \tilde{\epsilon}\circ p_0^{-1}(q) = \tilde{\epsilon}(q,1,\pi(q(a-c\alpha)/c) = \pi(q/c)$,

\item $(n,z)\cdot q = v(n,1,z) \cdot p_0^{-1}(q) = (n+q, 1, z)$, given $z = \pi((q+n)(a-c\alpha)/c)$ and

\item $q\cdot (n,z) = p_0^{-1}(q) \cdot h(n(a-c\alpha)^{-1},z,1) = (q + n(a-c\alpha)^{-1}, 1, \mu(q))$, given $z = \pi(q/c)$.
\end{itemize}
This concludes the proof.

\end{proof}

\begin{remark}
If $M = 
  \begin{pmatrix}
    a & b  \\
    c & d 
  \end{pmatrix}$
  in $SL_2(\ZZ[1/p])$ is such that $c=0$ then, with the notations above, $\beta=\epsilon \alpha+ n$ with $\epsilon \in \ZZ[1/p]^* = \pm p^\ZZ$ (i.e. invertible) and $n\in \ZZ[1/p]$. Thus $S_\beta=S_{\epsilon \alpha}$ and as we will check below, $S_{\epsilon \alpha}$ is isomorphic to $S_\alpha$.
\end{remark}

\subsection{Extension to $GL_2(\ZZ[1/p])$ is trivial up to isomorphism}

Let $\tilde{M} = 
  \begin{pmatrix}
    a & b  \\
    c & d 
  \end{pmatrix}$
  in $GL_2(\ZZ[1/p])$ with $\det(\tilde{M})=\pm p^l = \epsilon$. 

  Write $\tilde{M} = M\cdot M_\epsilon$ with $M = 
  \begin{pmatrix}
    \epsilon^{-1}a & b  \\
    \epsilon^{-1}c & d 
  \end{pmatrix}$  and  $M_\epsilon = 
  \begin{pmatrix}
    \epsilon  & 0  \\
         0  & 1 
  \end{pmatrix}$.

  A direct check shows that the map $\mu_{\epsilon}: S(\alpha) \to S(\epsilon\alpha)$ given by $\mu_{\epsilon}(n,z) = (\epsilon^{-1}n, z)$ is an isomorphism of groupoids with inverse $\mu_{\epsilon^{-1}}$ which provides the Morita equivalence $M_\epsilon\cdot\alpha \simeq \alpha$.

       Composition with the equivalence bundle $(\Sol_\alpha)_V^{H_M}$ of previous section associated with $M \in SL_2(\ZZ[1/p])$ 
       gives the Morita equivalence $\beta = \tilde{M}\cdot\alpha = M\cdot (M_\epsilon\cdot \alpha) \sim M_\epsilon\cdot\alpha \simeq \alpha$.

       Thus $M$ and $\tilde{M}$ define the same Morita equivalence class up to the isomorphism $\epsilon\alpha \simeq \alpha$ and 
       equivalence classes of noncommutative solenoids are found in the kernel of the exact sequence :
\[
  \xymatrix{
    SL_2(\ZZ[1/p]) \ar[r] &  GL_2(\ZZ[1/p]) \ar[r]^{\det} &  \ZZ[1/p]^* = \pm p^\ZZ
  }
\]

\begin{example} 
The case $\tilde{M} = \begin{pmatrix}    -1 & 0  \\ 0 & 1 \end{pmatrix}$ implies in particular that $-\alpha \simeq \alpha$.
\end{example} 

\begin{example} 
The case $\tilde{M} = \begin{pmatrix}    p^l & 0  \\ 0 & 1 \end{pmatrix}$ implies that $p^l\alpha \simeq \alpha$.
\end{example}

 \subsection{Equivalence bibundles of noncommutative torus}
 
In this section we treat, as an example of our computation methods, the case of noncommutative torus, or more precisely of groupoids behind the associated noncommutative algebras. This is perhaps a folklore and vastly addmited result, or at least for the associated $C^*$-algebras. However, in its groupoid counterpart we could not find any precise reference in the litterature, we give the result below. 
 
 Let $\theta \in \RR\setminus \QQ$.

Let $M = 
  \begin{pmatrix}
    a & b  \\
    c & d 
  \end{pmatrix}$
  in $SL_2(\mathbb{Z})$ and $\theta' = M^{-1}\cdot\theta = -\frac{b-d\theta}{a-c\theta}$ (where we assume $c\neq 0$ as for the solenoidal case). 

Denote $P_M$ the bibundle 

\begin{equation}
\xymatrix{
  *+[l]{\ZZ\ltimes_\theta S^1 } \ar@<.5ex>[d]\ar@<-.5ex>[d] & \RR \ar@{->}[ld]_\mu \ar[rd]^\epsilon & *+[r]{\ZZ\ltimes_{\theta'} S^1}  \ar@<.5ex>[d]\ar@<-.5ex>[d] \\
  S^1 & & S^1
}
\end{equation}
with:

Left and right moment maps are $\mu(t) = \pi(t(a-c\theta)/c)$ and $\epsilon(t) = \pi(t/c)$,

Left action is $(n,z)\cdot t = n+t$, given $z = \pi((t+n)(a-c\theta)/c)$

Right action is $t \cdot (n,z) = t + n(a-c\theta)^{-1}$, given $z=\pi(t/c)$.

Where in this case, $\pi:\RR\to S^1$ is the classic exponential map.

\begin{theorem}

  $P_M$ is an equivalence bibundle between $\mathbb{T}_\theta$ and $\mathbb{T}_{\theta'}$ 

\end{theorem}

\begin{proof}
 The proof is similar to the case of solenoids. The groupoid to be considered is $\gr_\theta = \RR\ltimes_{\theta}T^2$ the Kronecker groupoid defined by the flow of irrational slope $\theta$ over the torus $T^2$. The $\RR$-action is $t\cdot (x,y) = (\pi(t)x, \pi(t\theta)y)$ with $\RR/\ZZ\simeq_{\pi}S^1$ the canonical quotient map.

 The vertical and horizontal submanifolds $V=\{(1,y), y\in S^1\}$ and $H_M = \left \{ (\pi(ct), \pi(at)) \,,\, t \in \RR \right \}$ are closed transversals of $T^2$ and $(\gr_\theta)_V^{H_M} = \{(t,1,y), t\in \RR, y = \pi(t\theta)^{-1}\}$ is an equivalence bibundle between the reduced groupoids $(\gr_\theta)_V^V\simeq \ZZ\ltimes_\theta S^1 $ and $(\gr_\theta)_{H_M}^{H_M} \simeq \ZZ\ltimes_{\theta'} S^1 $.

\end{proof}

\section{Equivalence bimodules of noncommutative solenoids}

\subsection{Noncommutative solenoids}

In \cite{LP1} Packer and Latremoli\`ere observed that the group $H^2(\Gamma, \mathbb{T})$ of cohomologous multipliers of a commutative group $\Gamma$ is trivial for $\Gamma= \ZZ[1/p]$ but not for $\Gamma = \ZZ[1/p]^2$ since $H^2(\Gamma, \mathbb{T})$ is then isomorphic to $\mathscr{S}_p$ as a topological group. This result hence motivates the study of twisted $(\ZZ[1/p]^2, \sigma)$ groups and their $C^*$-algebras under the name of \textsl{noncommutative solenoids}.

\begin{definition}
  A \textsl{noncommutative solenoid} is a twisted group $C^*$-algebra $C^*(\Gamma, \sigma)$ for $\sigma$ a multiplier of the group $\Gamma = \ZZ[1/p]\times\ZZ[1/p]$.  
\end{definition}
 
It is the $C^*$-completion of the involutive Banach algebra $(l^1(\Gamma), *_\sigma, *)$ with twisted convolution $*_\sigma$ given for any $f_1,f_2\in l^1(\Gamma)$ by
$$ f_1 *_\sigma f_2 \,: \gamma\in\Gamma \rightarrow \sum_{\gamma_1\in \Gamma} f_1(\gamma_1)f_2(\gamma-\gamma_1)\sigma(\gamma_1, \gamma-\gamma_1)$$

and adjoint operation 
$$ f_1^* \,: \gamma\in\Gamma \rightarrow \overline{\sigma(\gamma, -\gamma)f_1(-\gamma_1)}$$

Cohomologous multipliers $\sigma$ and $\eta$ induce $*$-isomorphic $C^*$-algebras $C^*(\Gamma,\sigma)$ and $C^*(\Gamma, \eta)$.

There is another equivalent description of $H^2(\ZZ[1/p]^2, \mathbb{T}) \simeq \mathscr{S}_p$ as the additive group 
$$\Sigma_p = \{ (a_n) : a_0\in [0, 1) \text{ and } \left ( \forall n\in \NN, \exists k\in \{0,\dots,p-1\} ,\; pa_{n+1} = a_n + k \right ) \}$$
so that each $a\in\Sigma_p$ is in bijection with the class of the multiplier $\Psi_a$: 
$$( \frac{m_1}{p^{k_1}},  \frac{m_2}{p^{k_2}},  \frac{m_3}{p^{k_3}},  \frac{m_4}{p^{k_4}} )  \to \exp \left ({2i\pi a_{k_1+k_4}m_1m_4} \right )$$
and $\Psi_a$ and $\Psi_b$ are cohomologous if and only if $a = b\in\Sigma_p$. 

\begin{nota}
   The C*-algebra $C^*(\ZZ[1/p]^2, \Psi_a)$ is denoted $\mathscr{A}_a^\mathscr{S}$ by Latremoli\`ere and Packer.
\end{nota}
   

Noncommutative solenoids are interesting examples of noncommutative spaces, related to foliations, hyperbolic dynamical systems and wavelet analysis. 
Their $K_0$ groups are exactly the abelian extensions of $\ZZ[1/p]$ by $\ZZ$ and the range of the unique trace $\tau$ on $K_0(\mathscr{A}_\alpha^\mathscr{S})$ is $\ZZ\oplus_{k\in\NN}\alpha_k\ZZ$, $K_0$ groups of noncommutative solenoids thus being characterized by a $p$-adic integer modulo an integer.


 \subsection{The groupoid $C^*$-algebra $C^*(S_\alpha)$ as a noncommutative solenoid}

 Let $a = (e^{2i\pi a_k})_{k\in\NN} \in \Sp$ and consider the universal algebra generated by two families of commutative unitaries $(\cU_k)$ and $(\cV_l)$ :

$$\cA_a^\mathscr{S} = \left < \cU_k, \cV_l , k, l \in \NN \; | \; \cU_k \cV_l = e^{2i\pi a_{k + l}} \cV_l \cU_k \;,\; \cU_{k+1}^p = \cU_k \;,\; \cV_{l+1}^p = \cV_l \right > $$ 

Latremoli\`ere and Packer showed (\cite{LP1}, proposition 3.3, \cite{Zelprod}) that $\mathscr{A}_a^\mathscr{S}$ is $*$-isomorphic to both $\cA_a^\mathscr{S}$ and to the cross-product algebra $C(\Sp)\rtimes\ZZ[1/p]$ for the action of $\ZZ[1/p]$ on $\Sp$ defined as 
$$(\frac{n}{p^l} \cdot f)(z) = f((e^{2i\pi n a_{k + l}}z_k)_{k\in\NN})$$

We now prove that $\cA_a^\mathscr{S}$ can also be realized as the groupoid $C^*$-algebra $C^*(S_\alpha)$ for any  $\alpha\in\RR\times\Qp$ such that $\pi(\alpha)=a$.

\begin{theorem}\label{thmncsolenoidgrpd}
 $C^*(S_\alpha)$ and  $\cA_{\pi(\alpha)}^\mathscr{S}$ are $*$-isomorphic algebras.
\end{theorem}
\begin{proof}
For $k,l\in\NN$ let $U_k = \delta_{1/p^k}\otimes 1_{\Sp}$ and $V_l = \delta_0\otimes p_l$ with $\delta_x$ the Dirac at $x$ and $p_l:\Sp\rightarrow S^1$ the canonical $l$-th projection. $U_k$ and $V_l$ are both functions with compact support over $S_\alpha$ that satisfy the relations $U_{k+1}^p = U_k$ and $V_{l+1}^p = V_l$.
Let $\varphi : \cA_a^\mathscr{S} \To C^*(S_\alpha)$ be the map sending universal generators to those compact support functions :
$$\varphi(\cU_k) = U_k  \text{ and } \varphi(\cV_l) = V_l$$

We first check that $(U_k)_{k\in\NN}$ and $(V_l)_{l\in\NN}$ satisfy the universal property defining $\cA_a^\mathscr{S}$.
Their compositions are computed as $U_k * V_l (n,z) = p_l(z)\delta_{1/p^k}(n)$ and $V_l * U_k (n,z) = p_l(\pi(\frac{\alpha}{p^k})\cdot z)\delta_{1/p^k}(n)$ thus $U_k V_l = p_l(\pi(\frac{\alpha}{p^k}))V_l U_k$.

Let $t\in\RR$, $r\in\Qp$ such that $\alpha = (t,r)$.

Then $\pi(\frac{\alpha}{p^k}) = (e^{2i\pi\frac{t}{p^{k+n}}})_{n\in\NN}\cdot (e^{2i\pi\{\frac{-r}{p^{k+n}}\} })_{n\in\NN} = (e^{2i\pi a_{k+n}})_{n\in\NN}$

and $p_l(\pi(\frac{\alpha}{p^k})) = e^{2i\pi a_{k+l}}$ so that 
$$U_k V_l = e^{2i\pi a_{k+l}} V_l U_k.$$

The induced map $\varphi : \cA_a^\mathscr{S} \To C^*(S_\alpha)$ is thus a $*$-monomorphism.

To show that $\varphi$ can be extended to a $*$-epimorphism we now prove that $\overline{\im{\varphi}} = C^*(S_\alpha)$.

Let $i: C(\Sp) \hookrightarrow C_c(S_\alpha)$ be the $*$-morphism $i(f) = \delta_0 \otimes f$.
A direct computation shows that $i(f)*U_k^m = \delta_{m/p^k} \otimes ( \frac{m}{p^k} \cdot f )$.
For $f\in C_c(S_\alpha)$ and $n=m/p^k\in \ZZ[1/p]$ define $f_n \in C(\Sp)$ by $f_n(z) = f(n,z)$.
$f$ has compact support so there is a finite subset $I\subset \ZZ[1/p]$ such that
$$f = \sum_{n\in I} \delta_n \otimes f_n = \sum_{n=\frac{m}{p^k}\in I} i(-m/p^k\cdot f_n)*U_k^m$$
The set $A = \left \{ i(f)*U_k^n \; | \; f \in C_c(S_\alpha), k\in \NN, n\in \ZZ \right \}$ thus generates the algebra $C_c(S_\alpha)$.
The algebra generated by the canonical projections $< p_l^m, l, n\in\NN >$ is dense in $C(\Sp)$ by the Stone-Weierstrass theorem.
The relations $i(p_l^m) = V_l^m \in A$ and $\overline{<A>} = C^*(S_\alpha)$ then implies by continuity of $i$ that  $\left < U_k^n, V_l^m \; | \; k,l\in \NN, n,m\in \ZZ \right >$ is dense in  $C^*(S_\alpha)$.

$\varphi$ is thus both a $*$-monomorphism and a $*$-epimorphism and thus a $*$-isomorphism.

\end{proof}

 \subsection{Imprimitivity bimodules and the Equivalence Theorem}

 The content of Equivalence Theorem is informally that equivalent groupoids have Morita equivalent $C^*$-algebras.
 Let $Z$ be a $(G,H)$-equivalence where $G$ and $H$ are second countable locally compact Hausdorff groupoids with respective Haar systems $\lambda=\{\lambda^u\}_{u\in G^{(0)}}$ and $\mu=\{\mu^u\}_{u\in H^{(0)}}$. For $\varphi, \psi \in C_c(Z)$ consider the following $C_c(G)$- and $C_c(H)$-valued inner products :

\begin{equation}\label{inner_left}
	  \langle \varphi, \psi \rangle(\gamma) = \int_H \varphi(\gamma \cdot \omega \cdot \eta) \overline{\psi(\omega \cdot \eta)} d\mu^{s(z)}(\eta)
\end{equation}

 \begin{equation}\label{inner_right}
  \langle \varphi, \psi \rangle_*(\eta) = \int_G \overline{\varphi(\gamma^{-1}\cdot z)} \psi(\gamma^{-1}\cdot z \cdot \eta) d\lambda^{r(z)}(\gamma)
\end{equation}

and the $C_c(G)$- and $C_c(H)$-actions on $C_c(H)$ defined by

\begin{equation}\label{action_left}
  g \cdot \varphi(z) = \int_G g(\gamma) \varphi(\gamma^{-1} \cdot z) d\lambda^{r(z)}(\gamma)
\end{equation}

\begin{equation}\label{action_right}
  \varphi \cdot h(z) = \int_H \varphi(z \cdot \eta) h(\eta^{-1}) d\mu^{s(z)}(\eta)
\end{equation}

\begin{theorem}[Equivalence Theorem 2.70 \cite{Williams}]\label{Rentheorem}
  If $Z$ is a second countable locally compact Hausdorff space which is a  $(G,H)$-equivalence, then with respect to the actions and inner products $\eqref{inner_right}, \eqref{inner_left}, \eqref{action_right}, \eqref{action_left}$ $C_c(Z)$ can be completed to a $C^*(G,\lambda) - C^*(H,\mu)$- imprimitivity bimodule $X=X_H^G$.
\end{theorem}

The equivalence theorem for reduced algebra states that $C_c(Z)$ can also be completed to a $C_r^*(G,\lambda) - C_r^*(H,\mu)$- imprimitivity bimodule $X_r$.
In that case the identity map on $C_c(Z)$ extends to a surjective bimodule map of $X$ onto $X_r$ which kernel $Y$ is an imprimitivity bimodule. Moreover $X_r$ is then isomorphic to the quotient imprimitivity bimodule $X/Y$.

\begin{theorem}
Let $M = 
  \begin{pmatrix}
    a & b  \\
    c & d 
  \end{pmatrix}$
  in $GL_2(\ZZ[1/p])$, $\alpha\in \Sp\setminus\{1\}$ and $\beta = M^{-1}\cdot\alpha = -\frac{b-d\alpha}{a-c\alpha}$. 

$C_c(\RR\times\Qp)$ can be completed to a $C^*(S_\alpha)$ - $C^*(S_\beta)$ imprimitivity bimodule $E_M$ for which the algebras $\mathscr{A}_\alpha^\mathscr{S}$ and $\mathscr{A}_\beta^\mathscr{S}$ are Morita equivalent.
\end{theorem}

\begin{proof}
  $S_\alpha$ and $S_\beta$ are second countable locally compact Hausdorff groupoids.
The equivalence theorem is applied to the $(S_\alpha, S_\beta)$ equivalence bibundle $P_M$.  
\end{proof}

\bibliographystyle{plain}
\bibliography{bibliographie}

\end{document}